\documentclass[leqno,12pt]{amsart}
\usepackage{latexsym}
\usepackage{mathrsfs}
\usepackage{amsmath}
\usepackage{amssymb}
\usepackage[bookmarks]{hyperref}
\usepackage{pstricks,pst-plot}

\font\tenscr=rsfs10 % scaled \magstep1
\font\sevenscr=rsfs7 % scaled \magstep1
\font\fivescr=rsfs5 % scaled \magstep1
\skewchar\tenscr='177 \skewchar\sevenscr='177 \skewchar\fivescr='177
\newfam\scrfam \textfont\scrfam=\tenscr \scriptfont\scrfam=\sevenscr
\scriptscriptfont\scrfam=\fivescr

\font\tenscr=rsfs10 % scaled \magstep1
\font\sevenscr=rsfs7 % scaled \magstep1
\font\fivescr=rsfs5 % scaled \magstep1
\skewchar\tenscr='177 \skewchar\sevenscr='177 \skewchar\fivescr='177
\newfam\scrfam \textfont\scrfam=\tenscr \scriptfont\scrfam=\sevenscr
\scriptscriptfont\scrfam=\fivescr

\newtheorem{theorem}{Theorem}[section]

\newtheorem{corollary}[theorem]{Corollory}
\newtheorem{lemma}[theorem]{Lemma}

\newtheorem{example}{Example}

\newtheorem{definition} [theorem]{Definition}

\newcommand{\bthm}{\begin{theorem}}
\newcommand{\ethm}{\end{theorem}}
\newcommand{\blem}{\begin{lemma}}
\newcommand{\elem}{\end{lemma}}
\newcommand{\bcor}{\begin{corollary}}
\newcommand{\ecor}{\end{corollary}}
\newcommand{\bprop}{\begin{proposition}}
\newcommand{\eprop}{\end{proposition}}
\newcommand{\bdefn}{\begin{definition}}
\newcommand{\edefn}{\end{definition}}
\newcommand{\bpf}{\begin{proof}}
\newcommand{\epf}{\end{proof}}

\newcommand{\bi}{\begin{itemize}}
\newcommand{\ei}{\end{itemize}}

\newcommand{\bc}{\begin{cases}}
\newcommand{\ec}{\end{cases}}

\newcommand{\ba}{\begin{array}}
\newcommand{\ea}{\end{array}}

\newcommand{\be}{\begin{equation}}
\newcommand{\ee}{\end{equation}}

\newcommand{\bea}{\begin{eqnarray}}
\newcommand{\eea}{\end{eqnarray}}

\newcommand{\beaa}{\begin{eqnarray*}}
\newcommand{\eeaa}{\end{eqnarray*}}

\newcommand{\beastar}{\begin{eqnarray*}}
\newcommand{\eeastar}{\end{eqnarray*}}

\def \ma {\mathfrak{M}_A} 
\def \mb {\mathfrak{M}_B} 
\def \mac {\mathfrak{M}_{A^C}}
 
\def \C {\mathbb C}
\def \CN {\mathbb C^N}
\def\R {\mathbb R}
\def \N {\mathbb N}
\def \bs{\setminus}
\def\h#1{\widehat {#1}}
\def\hh#1{\widehat {#1} \bs  {#1}}

\def\hr#1{h_r({#1})}

\def\tX{\widetilde X}

\def\oD{\overline{D}}
\def\pD{\partial{D}}

\def\sF{{\mathscr F}}

\def\Bzero#1{{\mathscr B}_0(#1)}

\def\forallBzero#1{\hbox{for all } f\in {{\mathscr B}_0(#1)}}

\def\textfrac#1#2{{\textstyle \frac{#1}{#2}}}

\def\od{\overline D}
\def\disc#1#2#3{D(#1_{#3},#2_{#3})} 
\def\odisc#1#2#3{{\overline D}(#1_{#3},#2_{#3})}

\def\od{\overline D}
\def\disc#1#2#3{D(#1_{#3},#2_{#3})} 
\def\odisc#1#2#3{{\overline D}(#1_{#3},#2_{#3})}

\begin{document}

\title[Topology of Gleason Parts]{Topology of Gleason Parts\\ in maximal ideal spaces\\ with no analytic discs}
%\title[Topology of Gleason Parts]{Topology of Gleason Parts\\ in the absence of analytic discs}
\author{Alexander J. Izzo}
\address{Department of Mathematics and Statistics, Bowling Green State University, Bowling Green, OH 43403}
\thanks{The first author was partially supported by a Simons collaboration grant and by NSF Grant DMS-1856010.}
\email{aizzo@bgsu.edu}
\author{Dimitris Papathanasiou}
\thanks{The second author was supported by the Fonds de la Recherche Scienti que-FNRS, grant no. PDR T.0164.16}
\address{Institute de Mathématique,
Université de Mons,
7000 Mons, Belgium.}	
\email{dpapath@bgsu.edu}

\begin{abstract}
We strengthen, in various directions, the theorem of Garnett that every 
$\sigma$-compact, completely regular space $X$ occurs as a Gleason part for some uniform algebra.  In particular, we show that the uniform algebra can always be chosen so that its maximal ideal space contains no analytic discs.  We show that when the space $X$ is metrizable, the uniform algebra can be chosen so that its maximal ideal space is metrizable as well.  We also show that for every locally compact subspace $X$ of a Euclidean space, there is a compact set $K$ in some $\C^N$ so that $\hh K$ contains a Gleason part homeomorphic to $X$ and $\h K$ contains no analytic discs.
\end{abstract}

\maketitle

 \section{Introduction}
 Early in the study of uniform algebras it was conjectured that whenever the maximal ideal space $\ma$ of a uniform algebra $A$ on a compact Hausdorff space $K$ is strictly larger than $K$, then the complementary set $\ma\setminus K$ must contain an analytic disc.  This conjecture was Andrew Gleason's motivation for introducing his notion of {\em parts\/} in a maximal ideal space~\cite{Gleason}.  His hope was that each nontrivial part would be a union of analytic discs.  However, the conjecture about the existence of analytic discs in maximal ideal spaces turned out to be false as shown by Gabriel Stolzenberg~\cite{Stol1}.  Furthermore, it was later shown by John Garnett that in general, nontrivial Gleason parts need not be, in any reasonable sense, analytic sets.  Specifically Garnett proved the following striking theorem~\cite{Garnett}.
 
 \bthm\cite[Theorem~2]{Garnett}\label{Garnett-thm}
 Let $X$ be a $\sigma$-compact, completely regular space.  Then there exists a uniform algebra $B$ and a Gleason part $P$ for $B$ homeomorphic to $X$ such that the restriction $B|P$ is isometrically isomorphic to the algebra $C_b(X)$ of all bounded continuous functions on $X$. 
\ethm

Since it is easily seen that each Gleason part must be a $\sigma$-compact, completely regular space, this theorem characterizes the spaces that occur as Gleason parts.  Furthermore, the theorem decisively shows that Gleason parts need not contain analytic discs.  Nevertheless, there are analytic discs in the maximal ideal spaces of the uniform algebras constructed by Garnett.  Thus the theorem leaves open the possibility that the presence of a particular space as a Gleason part might force the existence of analytic discs in a maximal ideal space.  The main purpose of the present paper is to show that, on the contrary, every $\sigma$-compact, completely regular space occurs as a Gleason part for some uniform algebra whose maximal ideal space contains no analytic discs.

There are additional issues left unsettled by Garnett's theorem that we will address.  The maximal ideal space of the uniform algebra in Theorem~\ref{Garnett-thm} is enormous whenever the space $X$ is noncompact since it necessarily contains the Stone-\v Cech compactification of $X$.  Given a $\sigma$-compact, {\em metrizable} space $X$, one would like to get $X$ as a Gleason part in a {\em metrizable} maximal ideal space.  We will show that this can indeed be done, thus characterizing the spaces that occur as Gleason parts in metrizable maximal ideal spaces.  (In the case when $X$ is $\sigma$-compact and
{\em \hbox{locally} compact}, Garnett \cite[Corollary]{Garnett} gave a second construction of a uniform algebra with a Gleason part homeomorphic to $X$, simpler than the proof of Theorem~\ref{Garnett-thm}.  For $X$ a {\em locally compact}, $\sigma$-compact, metrizable space, this simpler construction yields a uniform algebra whose maximal ideal space is metrizable although that is not pointed out in Garnett's paper.)  In the case when $X$ embeds in a Euclidean space, one would like to get $X$ as a Gleason part in the maximal ideal space of a finitely generated uniform algebra.  Here the issue can be rephrased more concretely in terms of polynomial hulls as follows: Given a $\sigma$-compact space $X$ that embeds in a Euclidean space, does there exist a compact set $K$ in some $\C^N$ such that the polynomial hull $\h K$ of $K$ contains a space $P$  homeomorphic to $X$ and such that $P$ is a Gleason part for the uniform algebra $P(K)$?  Under the additional hypothesis that $X$ is locally compact, we will answer this affirmatively.  This fails to provide a characterization of the topological spaces that occur as Gleason parts in polynomial hulls, for as we will show, there exist compact planar sets $K$ such that the uniform algebra $R(K)$ has a Gleason part that is not locally compact.  The question of whether every $\sigma$-compact subspace of a Euclidean space occurs as a Gleason part in a polynomial hull remains open.

Of course when we require our Gleason part $P$ to lie in a metrizable space, or a Euclidean space, we must give up the condition in Garnett's theorem that the restriction $B|P$ is isometrically isomorphic to $C_b(X)$.  
We will show, however, that in the general case, 
we can make the restriction $B|P$ isometrically isomorphic to the algebra of continuous functions determined by {\em any} compactification of $X$, while when we require $P$ to lie in a metrizable space, or a Euclidean space, the compactification must of course also be restricted to lie in such a space, but there are no further restrictions.

Another issue is that because the original conjecture about analytic structure concerned analytic discs in the complementary set $\ma\bs K$, it is of interest to obtain our Gleason part in this complementary set.  We will indeed achieve this.   Actually Garnett's original construction yielded the Gleason part in this complementary set as well although he did not remark upon this. 
% (DIMITRIOS, PLESE CHECK THAT THIS IS INDEED THE CASE.  I DO NOT CURRENTLY HAVE ACCESS TO GARNETT'S %PAPER TO CHECK.)

The precise statements of our main theorems are as follows.

\begin{theorem}\label{T:1}
Let $X$ be a $\sigma$-compact, completely regular space.  Then there exists a uniform algebra $B$ on a compact Hausdorff  space $K$ with a Gleason part $P$ homeomorphic to $X$ contained in $\mb\bs K$ and  such that $\mb$ contains no analytic discs.  Furthermore, if $\tX$ is any compactification of $X$, then $B$ and $K$ can be chosen so that the restriction $B|P$ is isometrically isomorphic to $C(\tX)$.
\ethm

\begin{theorem}\label{T:2}
Let $X$ be a $\sigma$-compact, metrizable space.  Then there exists a uniform algebra $B$ on a compact, metrizable space $K$ with a Gleason part $P$ homeomorphic to $X$ contained in $\mb\bs K$ and  such that $\mb$ contains no analytic discs.  Furthermore, if $\tX$ is any metrizable compactification of $X$, then $B$ and $K$ can be chosen so that the restriction $B|P$ is isometrically isomorphic to $C(\tX)$.
\ethm

\begin{theorem}\label{T:3}
Let $X$ be a locally compact subspace of $\R^m$.  Then there exists a compact set $K$ in $\C^{2m+2}$ such that $P(K)$ has a Gleason part $P$ homeomorphic to $X$ contained in $\hh K$ and such that $\h K$ contains no analytic discs.  Furthermore, if $\tX$ is any compactification of $X$ contained in $\R^m$, then $K$ can be chosen so that the restriction $P(K)|P$ is isometrically isomorphic to $C(\tX)$.
\ethm

In connection with Theorem~\ref{T:3}, note that every locally compact subspace of a Euclidean space is $\sigma$-compact.  
It should also be noted that by the theorem of Menger and N\"obeling about embedding spaces of finite topological dimension in Euclidean spaces \cite[Theorem~V.2]{HW}, the locally compact subspaces of $\R^m$ are precisely the locally compact, separable, metrizable spaces of finite topological dimension. 

Given that the introduction of Gleason parts was motivated by the search for analytic structure, it is surprising that until recently little attention was given in the literature to the existence of Gleason parts in maximal ideal spaces containing no analytic discs.  The first result concerning Gleason parts in this setting is due to Richard Basener~\cite{Basener} who proved that there exists a compact set in $\C^2$ with a nontrivial rational hull that contains no analytic discs but contains a Gleason part of positive 
4-dimensional measure.  From this it follows easily, by a standard device recalled at the end of the next section, that there exists a compact set in $\C^3$ with a nontrivial {\it polynomial\/} hull that also contains no analytic discs but contains a Gleason part of positive 
4-dimensional Hausdorff measure.  Further examples of nontrivial polynomial and rational hulls without analytic discs and containing Gleason parts that are large in the sense of Hausdorff measure were given recently by the first author of the present paper in \cite{Izzo} and \cite{Izzo2}.
Brian Cole~\cite{Cole} gave a construction that produces a uniform algebra $A$ defined on a proper subset of its maximal ideal space $\ma$ such that 
$\ma$ contains {\em no} nontrivial Gleason parts \cite{Cole}.  Fifty years later it was shown in the paper \cite{CGI} of Cole, Swarup Ghosh, and the first author of the present paper, that there exists a nontrivial polynomial hull in $\C^3$ containing no nontrivial Gleason parts.  

In Garnett's proof of Theorem~\ref{Garnett-thm}, first one constructs a certain uniform algebra $A$ whose maximal ideal space contains a Gleason part that contains a subspace homeomorphic to $X$.  The uniform algebra $B$ whose existence is asserted by the theorem is then obtained from $A$.  The proofs of Theorems~\ref{T:1}, \ref{T:2}, and~\ref{T:3} all have this same overall structure.  However, substantial changes are needed.  Curiously, in the proofs of Theorems~\ref{T:1} and~\ref{T:2}, the substantial changes are in the second step, while in the proof of Theorem~\ref{T:3} they are in the first step.  In the proofs of Theorems~\ref{T:1} and~\ref{T:2}, we construct the initial uniform algebra as in Garnett's proof.  The maximal ideal space of this uniform algebra contains analytic discs.  So as to eliminate the analytic discs, we carry out the passage from $A$ to $B$ using Coles's method of root extensions~\cite{Cole}, a completely different approach from the one used by Garnett.  (Note that since Theorem~\ref{T:1} contains Theorem~\ref{Garnett-thm}, this gives an alternative proof of Garnett's theorem.)  To prove Theorem~\ref{T:3}, one must obtain at the end a finitely generated uniform algebra.  The authors do not see a way to carry out the passage from $A$ to $B$ using the Cole construction so as to achieve this.  Therefore, in the proof of Theorem~\ref{T:3}, we change the construction of the initial uniform algebra so as to obtain at the outset a uniform algebra $A$ whose maximal ideal space contains no analytic discs.  We then carry out the passage from $A$ to $B$ by a modification of the argument used by Garnett.  The authors have not been able to carry out the construction of the initial uniform algebra $A$ used in the proof of Theorem~\ref{T:3} in the case when $X$ is not locally compact.  Nevertheless, the approach used in the proof of Theorem~\ref{T:3} can be used to give an alternative proof of the special case of Theorems~\ref{T:1} and~\ref{T:2} when $X$ is locally compact.

In the next section we recall some definitions and notations already used above.  
Theorems~\ref{T:1} and~\ref{T:2} are proven in Section~3.  Section~4 contains the proof of Theorem~\ref{T:3} along with examples showing that the theorem does not provide a full characterization of the spaces that occur as Gleason parts in polynomial hulls.

%%%%%%%%%%%%%%%%%%%%%%%%%%%%%%%%%%%%%%%%%
%                                                    PRELIMINARIES
%
%%%%%%%%%%%%%%%%%%%%%%%%%%%%%%%%%%%%%%%%

\section{Preliminaries}~\label{prelim}

For
$K$ a compact Hausdorff space, we denote by $C(K)$ the algebra of all continuous complex-valued functions on $K$ with the supremum norm
$ \|f\|_{K} = \sup\{ |f(x)| : x \in K \}$.  A \emph{uniform algebra} $A$ on $K$ is a closed subalgebra of $C(K)$ that contains the constant functions and separates
the points of $K$.  We tacitly regard $K$ as a subspace of the maximal ideal space $\ma$ of $A$ by identifying each point of $K$ with the corresponding point evaluation functional.  When convenient, we will also tacitly regard $A$ as a uniform algebra on $\ma$ via the Gelfand transform.

For a compact set $K$ in $\CN$, the \emph{polynomial hull} $\h K$ of $K$ is defined by
$$\h K=\{z\in\CN:|p(z)|\leq \max_{x\in K}|p(x)|\
\mbox{\rm{for\ all\ polynomials}}\ p\},$$
and the
\emph{rational hull} $\hr K$ of $K$ is defined by
$$\hr K = \{z\in\C^N: p(z)\in p(K)\ 
\mbox{\rm{for\ all\ polynomials}}\ p
\}.$$
%An equivalent formulation of the definition of $\hr K$ is that $\hr K$ consists precisely of those points $z\in \C^N$ such that every polynomial that vanishes at $z$ also has a zero on $K$.
The set $K$ is said to be \emph{polynomially convex} if $\h K = K$ and 
\emph{rationally convex} if $\hr K =K$.  We say that a polynomial hull $\h K$ (or rational hull $\hr K$) is \emph{nontrivial} if $\h K \neq K$ (or $\hr K \neq K$).

We denote by 
$P(K)$ the uniform closure on $K\subset\CN$ of the polynomials in the complex coordinate functions $z_1,\ldots, z_N$, and we denote by $R(K)$ the uniform closure of the rational functions  holomorphic on (a neighborhood of) $K$. 
Both $P(K)$ and $R(K)$ are uniform algebras, and
it is well known that the maximal ideal space of $P(K)$ can be naturally identified with $\h K$, and the maximal ideal space of $R(K)$ can be naturally identified with $\hr K$.

Throughout the paper, $D$ will denote the open unit disc in the complex plane, and $A(D)$ will denote the disc algebra, that is, the uniform algebra of continuous functions on $\oD$ that are holomorphic functions on $D$.

By an \emph{analytic disc} in $\CN$, we mean an injective holomorphic map $\sigma: D\rightarrow\CN$.
By the statement that a subset $S$ of $\CN$ contains no analytic discs, we mean that there is no analytic disc in $\CN$ whose image is contained in $S$.
An \emph{analytic disc} in the maximal ideal space $\ma$ of a uniform algebra $A$ is, by definition, an injective map $\sigma:D\rightarrow \ma$ such that the function $\hat f\circ \sigma$ is analytic on $D$ for every function $f$ in 
$A$, where $\hat f$ denotes the Gelfand transform of $f$.

Let $A$ be a uniform algebra on a compact Hausdorff space $K$.
The \emph{Gleason parts} for the uniform algebra $A$ are the equivalence classes in the maximal ideal space of $A$ under the equivalence relation $\varphi\sim\psi$ if $\|\varphi-\psi\|<2$ in the norm on the dual space $A^*$.  (That this really is an equivalence relation is well-known but {\it not\/} obvious!)
We say that a Gleason part is \emph{nontrivial} if it contains more than one point.
It is immediate that the presence of an analytic disc in the maximal ideal space of $A$ implies the existence of a nontrivial Gleason part.

The following two lemmas are standard.  (See \cite[Lemmas~2.6.1 and~2.6.2]{Browder}.)

\begin{lemma}\label{samepart}
Two multiplicative linear functionals $\phi$ and $\psi$ on a uniform algebra $A$ lie in the same Gleason part if and only if
$$\sup\{|\psi(f)|: f\in A, \|f\|\leq 1, \phi(f)=0\}< 1.$$
\end{lemma}

\begin{lemma}\label{peak}
Let $\phi$ and $\psi$ be two multiplicative linear functionals on a uniform algebra $A$.  If there is a function $f\in A$ such that $\|f\|=1$, $|\phi(f)|<1$, and $\psi(f)=1$, then $\phi$ and $\psi$ lie in different Gleason parts.
\end{lemma}

A subset $S$ of the maximal ideal space of a uniform algebra $A$ is said to be a {\it hull\/} if $S$ is the common zero set of some collection of Gelfand transforms of elements of $A$.  In other words, $S\subset\ma$ is a {\it hull\/} if 
\[ S=\{ \phi\in \ma: \hat f(\phi)=0 \ {\rm for\ all\ } f\in A \ {\rm such\ that\ } \hat f|S=0\}. \]
This use of the term {\it hull\/} should not be confused with the use of the word hull in the terms {\it polynomial hull\/} and {\it rational hull\/}.
The ideal of functions in $A$ whose Gelfand transforms vanish on $S$ will be denoted by $I(S)$.

By a compactification of a space $X$ we mean a compact Hausdorff space that contains a dense subspace that is homeomorphic to $X$.  

Given a set $\Omega$ in $\CN$, we denote the topological boundary of $\Omega$ by $\partial \Omega$.
We denote the set of positive integers by $\N$.

Finally we recall the standard method for showing that every uniform algebra of the form $R(L)$ for $L$ a compact planar set is isomorphic as a uniform algebra to 
$P(K)$ for some compact set $K$ in $\C^2$.
A theorem due to Kenneth Hoffman and Errett Bishop (and in more general form Hugo Rossi~\cite{Rossi}) asserts that for $L$ a compact set in $\C$, there is a function $g$ such that the identity function $z$ and $g$ generate $R(L)$ as a uniform algebra.  Let $\tau:L\rightarrow \C^2$ be given by $\tau(z)=\bigl(z, g(z)\bigr)$.  Then $P(\tau(L))$ is isomorphic as a uniform algebra to $R(L)$.

%%%%%%%%%%%%%%%%%%%%%%%%%%%%%%%%
%%%%%%%%%%%%%%%%%%%%%%%%%%%%%%%%

% Section 
%%%%%%%%%%%%%%%%%%%%%%%%%%%%%%%%
%%%%%%%%%%%%%%%%%%%%%%%%%%%%%%%%

\section {Gleason parts in maximal ideal spaces\\ without analytic discs}\label{general-case}

In this section we prove Theorems~\ref{T:1} and~\ref{T:2}.  As discussed in the introduction, we will achieve the absence of analytic discs by using Cole's method of root extensions \cite{Cole}.  In Cole's original application of the method, he repeatedly adjoined square roots for every function in a nontrivial uniform algebra and thereby obtained a nontrivial uniform algebra with no nontrivial Gleason parts.  We will adjoin square roots only to functions that vanish on a fixed hull and thereby eliminate nontrivial Gleason parts outside that hull while preserving nontrivial Gleason parts in the hull.
Our use of Cole's method of root extensions is contained in the proof of the following theorem which will be invoked in the proofs of Theorems~\ref{T:1} and~\ref{T:2}.  Some earlier applications of Cole's method using only functions vanishing on a 
specific set were given by Joel Feinstein \cite{Feinstein1}, \cite{Feinstein2}, \cite{Feinstein3} and Feinstein and Matthew Heath \cite{FeinsteinH}.

Recall that given a uniform algebra $A$ and a subset $S$ of $\ma$, we use the notation $I(S)=\{f\in A: \hat f|S=0\}$.

\begin{theorem}\label{conjecture}
Let $A$ be a uniform algebra on a compact Hausdorff space $Y$, and let $S\subset \ma$ be a hull for $A$.  Then there exists a uniform algebra $A^C$ on a compact Hausdorff space $Y^C$ and a continuous surjective map $\pi:\mac\rightarrow \ma$ such that 
\item{\rm(i)} $\pi(Y^C)=Y$,
\item{\rm(ii)} the formula $\pi^*(f)=f\circ \pi$ defines an isometric embedding of $A$ into $A^C$,
\item{\rm(iii)} $\pi$ takes $\pi^{-1}(S)$ homeomorphically onto $S$,
\item{\rm(iv)} $\pi^{-1}(S)$ is a hull for $A^C$,
\item{\rm(v)} $\pi^*$ induces in the obvious way an isometric isomorphism of $A|S$ onto $A^C|\pi^{-1}(S)$,
\item{\rm(vi)} the set $\{f^2: f \in I(\pi^{-1}(S))\}$ is dense in $I(\pi^{-1}(S))$, and
\item{\rm(vii)} the nontrivial Gleason parts for $A^C$ are exactly the sets of the form $\pi^{-1}(S\cap \Pi)$ for $\Pi$ a Gleason part for $A$ such that $S\cap \Pi$ contains more than one point.
\hfil
\break
Furthermore, if the maximal ideal space of $A$ is metrizable, then $A^C$ can be chosen so that its maximal ideal space is metrizable as well.
\ethm

\bpf[Proof of Theorem~\ref{conjecture}]
\hfil

\emph{Step 1}:  We construct the the uniform algebra $A^C$.

Let $\Sigma_0=\ma$, and let $A_0$ denote $A$ regarded as a uniform algebra on $\Sigma_0$.  Let $S_0=S$.  We will define a sequence of uniform algebras $\{A_m\}_{m=0}^\infty$.  
First let $\sF_0$ be a dense subset of $I(S_0)$.  In case the maximal ideal space of $A$ is metrizable, choose $\sF_0$ to be countable.  (In the general case, one can take $\sF_0=I(S_0)$.)  Let $p_1:\Sigma_0\times \C^{\sF_0}\rightarrow \Sigma_0$ and $p_f:\Sigma_0\times \C^{\sF_0} \rightarrow \C$ denote the projections given by $p_1\bigl(x,(z_g)_{g\in \sF_0}\bigr)=x$ and $p_f\bigl(x,(z_g)_{g\in \sF_0}\bigr)=z_f$.  Define $\Sigma_1\subset \Sigma_0\times \C^{\sF_0}$ by
$$\Sigma_1=\{y\in \Sigma_0\times \C^{\sF_0}: (p_f(y))^2=f(p_1(y)) \hbox{\ for all } f\in \sF_0\},$$
and let $A_1$ be the uniform algebra on $\Sigma_1$ generated by 
the set of functions $\{f\circ p_1: f\in A_0\} \cup \{p_f: f\in \sF_0\}$.
On $\Sigma_1$ we have $p^2_f=f\circ p_1$ for every $f\in \sF_0$.
Set $\pi_1=p_1|\Sigma_1$, and note that $\pi_1$ is surjective.  Also there is an isometric embedding $\pi^*_1:A_0\rightarrow A_1$ given by $\pi^*_1(f)=f\circ\pi_1$.  Let $S_1=\pi^{-1}_1(S_0)$.  Then $S_1=S_0\times \{0\}^{\sF_0}$, so $\pi_1$ takes $S_1$ homeomorphically onto $S_0$.  Also observe that $S_1$ is a hull for $A_1$.  In case $\Sigma_0$ is metrizable, so is $\Sigma_1$.  By \cite[Theorem~1.6]{Cole}, the maximal ideal space of $A_1$ is $\Sigma_1$.

We next iterate the construction to obtain a sequence\break $\{(A_m,\Sigma_m, \pi_m, S_m, \sF_m)\}_{m=0}^\infty$, where each $A_m$ is a uniform algebra on $\Sigma_m$, each $\pi_m:\Sigma_m\rightarrow \Sigma_{m-1}$ is a surjective continuous map, each $S_m$ is a hull for $A_m$ such that $\pi_m$ takes $S_m$ homeomorphically onto $S_{m-1}$, and each $\sF_m$ is a dense subset of $I(S_m)$ such that for every $f\in \sF_m$ the function $f\circ \pi_{m+1}$ is the square of a function in $A_{m+1}$.  Finally we take the inverse limit of the system of uniform algebras $\{A_m\}$.  Explicitly, we set 
$$\Sigma_\omega=\Bigl\{(y_k)_{k=0}^\infty\in \textstyle\prod\limits_{k=0}^\infty \Sigma_k: \pi_{m+1}(y_{m+1})=y_m  \hbox{\ for all } m=0, 1, 2, \ldots\Bigr\},$$
and letting $q_m:\Sigma_\omega\rightarrow \Sigma_m$ be the restriction of the canonical projection $\prod_{k=0}^\infty \Sigma_k \rightarrow \Sigma_m$, we let $A_\omega$ be the uniform closure of $\bigcup_{m=0}^\infty \{h\circ q_m: h\in A_m\}$.  Note that the sets $q_m^{-1}(S_m)$, as $m$ ranges over $\N\cup \{0\}$, all coincide.  Let $S_\omega$ denote this common set.  Set $\pi=q_0$.  Then of course $S_\omega=\pi^{-1}(S)$.

By \cite[Theorem~2.3]{Cole}, the maximal ideal space of $A_\omega$ is $\Sigma_\omega$, and the inverse image under $\pi$ of the Shilov boundary for $A_0$ is the Shilov boundary for $A_\omega$.  Consequently, setting $Y^C=\pi^{-1}(Y)$ and  $A^C$ equal to the restriction algebra $A_\omega|Y^C$, we have that $A^C$ is a uniform algebra isometrically isomorphic to $A_\omega$. 

Note that in case $\Sigma_0$ is metrizable, so is $\Sigma_\omega$.

\emph{Step 2}:  Observe that properties (i)--(iv) hold.  Note also that for each $m\in \N$, the formula $q_m^*(f)=f\circ q_m$ defines an isometric embedding of $A_m$ into $A_\omega$.

\emph{Step 3}:  We introduce certain norm 1 linear operators.  

Cole \cite[Theorem~2.1]{Cole} (see also \cite[Lemma~19.3]{Stout}) proved the existence of a continuous linear operator $T_0:A_\omega \rightarrow A_0$ of norm 1 such that $T_0\circ \pi^*$ is the identity on $A_0$.
When taking the inverse limit of a system of uniform algebras, discarding the first $m$ algebras, i.e., replacing the system $\{A_k\}_{k=0}^\infty$ by $\{A_k\}_{k=m}^\infty$, yields a uniform algebra isometrically isomorphic in an obvious way to $A_\omega$.  Consequently, we get for each $m$, an analogous norm 1 operator  $T_m:A_\omega\rightarrow A_m$ such that $T_m\circ q_m^*$ is the identity on $A_m$.
It is clear from the way that $T_0$ is defined, that for each point $x\in \Sigma_m$, the value of $(T_mf)(x)$ depends only on the values of $f\in A_\omega$ on the fiber $q_m^{-1}(x)$.  (The operator $T_0$ is obtained from a sequence of operators with the $n$th operator given by averaging over the fibers of the map $\pi_n\circ \cdots\circ \pi_1$.)  The fact that $T_m\circ q_m^*$ is the identity then gives that for any $f\in A_\omega$ and any point $x\in\Sigma_m$ such that $q_m^{-1}(x)$ consists of a single point, $(T_mf)(x)=f(q_m^{-1}(x))$.  Consequently,
$q_m^*(T_mf)|S_\omega=f|S_\omega$ for every $f\in A_\omega$. 

\emph{Step 4}: For the proof of (v), note that given $f\in A$, the restriction of $\pi^*(f)=f\circ \pi$ to $\pi^{-1}(S)=S_\omega$ depends only on the restriction of $f$ to $S$, so $\pi^*$ induces a map of $A|S$ into $A^C|\pi^{-1}(S)$ that is obviously isometric.  Setting $q=0$ in the last sentence of Step~3 yields the equality 
$\pi^*(T_0f)|\pi^{-1}(S)=f|\pi^{-1}(S)$ for all $f\in A_\omega$, which establishes that the map is onto.

\emph{Step 5}:  For the proof of (vi), first note that each element of $q_m^*(\sF_m)$ has a square root in $q_{m+1}^*(A_{m+1})$; indeed, given $h\in \sF_m$, there exists $k\in A_{m+1}$ such that $k^2=h\circ \pi_{m+1}$, and then $(k\circ q_{m+1})^2=h\circ q_m$.  Furthermore, $q_m^*(\sF_m)$ is dense in $q_m^*(I(S_m))$.  Consequently, to prove~(vi), it suffices to show that $\bigcup_{m=0}^\infty q_m^*(I(S_m))$ is dense in $I(S_\omega)$.  

Fix $f\in I(S_\omega)$ and $\varepsilon>0$ arbitrary.  We will show that 
$\|q_m^*(T_m f) - f\|<\varepsilon$ for some $m\in \N$.  Since $q_m^*(T_m f)$ is in $q_m^*(I(S_m))$ by the last sentence of Step~3, this will establish the desired density.

By the definition of $A_\omega$, there exist $m\in \N$ and $h\in A_m$ such that 
\begin{equation} 
\|f-q_m^*(h)\|<\varepsilon/ 2.
\label{1} \end{equation}
Then
\[
\|(q_m^*\circ T_m)(f) - (q_m^*\circ T_m)(q_m^* h) \|<\varepsilon/2.
\]
Since $T_m\circ q_m^*$ is the identity on $A_m$, this gives
\begin{equation} 
\|(q_m^*\circ T_m)(f) - (q_m^* h) \|<\varepsilon/2.
\label{2} \end{equation}
From~(\ref{1}) and~(\ref{2}) we get
\[
\bigl\|(q_m^*\circ T_m)(f) - f\bigr\| \leq \bigl\|(q_m^*\circ T_m)(f) - q_m^*(h)\bigr\| + \bigl \|q_m^*(h) - f\bigr\|
<\varepsilon.
\]

\emph{Step 6}:  Let $\|\cdot\|_0$ and $\|\cdot\|_\omega$
denote the dual space norms on $A_0^*\supset {\mathfrak{M}_{A_0}}$ and $A_\omega^*\supset {\mathfrak{M}_{A_\omega}}$, respectively.  We show that for $x$ and $y$ in $S_\omega$,
\[
\|x-y\|_\omega= \|\pi(x)-\pi(y)\|_0.
\]
Consequently, 
two points $x$ and $y$ of $S_\omega$ lie in the same Gleason part for $A_\omega$ if and  only if $\pi(x)$ and $\pi(y)$ lie in the same Gleason part for $A_0$.

By definition,
\[
\|x-y\|_\omega=\sup\bigl\{|h(x)-h(y)|: h\in A_\omega, \|h\|\leq 1\bigr\}
\]
and 
\[
\|\pi(x)-\pi(y)\|_0=\sup\bigl\{|f(\pi(x))-f(\pi(y))|:f\in A_0, \|f\|\leq 1\bigr\}.
\]
Because $f\circ\pi$ is in $A_\omega$ and satisfies $\|f\circ \pi\|\leq 1$ for every $f$ in $A_0$ with $\|f\|\leq 1$, we have 
\[
\|x-y\|_\omega\geq \|\pi(x)-\pi(y)\|_0.
\]
But given $h$ in $A_\omega$ with $\|h\|\leq 1$, the function $T_0h$ is in $A_0$ with $\|T_0h\|\leq 1$ and satisfies
$\bigl|(T_0h)(\pi(x))- (T_0h)(\pi(y))\bigr|=|h(x)-h(y)|$ by the end of Step~3.  Thus
we also have
\[
\|x-y\|_\omega\leq \|\pi(x)-\pi(y)\|_0.
\]

\emph{Step 7}: In view of (vi), 
Lemma~\ref{A_omega-parts} below gives that each point of $\Sigma_\omega\bs S_\omega$ is a one-point Gleason part for $A_\omega$.  Property (vii) is then a consequence of the result in Step 6.
\epf

\blem\label{A_omega-parts}
Let $A$ be a uniform algebra on a compact Hausdorff space $\Sigma$, and let $S$ be a hull for $A$.  If the set $\{f^2:f\in I(S)\}$ is dense in $I(S)$, then each point of $\Sigma\bs S$ is a one-point Gleason part for $A$.
\elem

\bpf
The proof is essentially a repetition of the proof of \cite[Lemma~1.1~(i)]{Cole}.
Let $x\in \Sigma\bs S$ be arbitrary, and let $y$ be an arbitrary element of $\Sigma$ distinct from $x$.  Since $S$ is a hull for $A$, one can find a function $f$ in $I(S)$ such that $f(x)\neq 0$, $f(y)=0$, and $\|f\|<1$.  For each $n\in \N$ and $\varepsilon>0$, there exist functions $f_1,\ldots, f_n$ in $I(S)$ such that
$\|f-f_1^2\|<\varepsilon, \ldots, \|f_{n-1}-f_n^2\|<\varepsilon$.  Choosing $\varepsilon>0$ small enough, $f_n^{2^n}$ can be made arbitrarily close to $f$.   Replacing $f_n$ by $f_n-f_n(y)$, this approximation can be done with $f_n(y)=0$.  Since $|f(x)|^{2^{-n}} \rightarrow 1$, choosing $n$ large enough, $|f_n(x)|$ can be made arbitrarily close to 1.  Thus $x$ and $y$ lie in different Gleason parts by Lemma~\ref{samepart}
\epf

%We are now ready for the proofs of Theorems~\ref{T:1} and~\ref{T:2}.

We are now ready to prove Theorems~\ref{T:1} and~\ref{T:2}.

\begin{proof}[Proof of Theorem~\ref{T:1}]
We treat first the case when $X$ is compact.  Recall that we denote the disc algebra by $A(D)$.  
Let 
\begin{equation*}
\begin{split}
A=\{f\in C(X \times \od): f|({\{x\}\times \od}) \in A(D)\ {\rm for\ all\ } x\in X \phantom{and}\\
 {\rm and\ } f|({X \times \{0\}})\ {\rm is\ constant} \}.
\end{split}
\end{equation*}
Then $\ma$ is the quotient space obtained from $X\times \od$ by collapsing the set $X \times \{ 0 \}$ to a point.  We will denote points of $\ma$ by their representatives in $X\times \od$.  Then the set  
$$
\Pi=\{(x,z)\in \ma: |z|<1 \}
$$
is a Gleason part for $A$. 

Let $S=\{(x, 1/2): x\in X\}$.  Then $S$ is a subspace of $\Pi$ homeomorphic to $X$, and $S$ is a hull since $S$ is the zero set of the function $g$ on $\ma$ defined by $g(x,z)=1-2z$.

The algebras $A|S$ and $C(X)$ are isometrically isomorphic.  To see this, let $\rho:X\rightarrow S$ be given by $\rho(x)=(x, 1/2)$, and note that given $f\in C(X)$, the function $l$  on $\ma$ defined by $l(y,z)=2zf(y)$ is in $A$, and $f=l\circ \rho$.  Consequently, the map $T:A \rightarrow C(X)$ given by $T(l)=l\circ\rho$
induces an isometric isomorphism of $A|S$ onto $C(X)$.

Every function in $A$ takes its maximum modulus on $X \times \partial D$.  (In uniform algebra terminology, $X \times \partial D$ is a boundary for $A$.)  Therefore, $A|(X \times \partial D)$ is a uniform algebra isometrically isomorphic to $A$.

Now we apply Theorem~\ref{conjecture} to the algebra 
$A|(X \times \partial D)$.  Letting $A^C$, $Y^C$, and $\pi$ denote the objects given by Theorem~\ref{conjecture}, set $B=A^C$, $K=Y^C$, and $P=\pi^{-1}(S)$.  Then $P$ is homeomorphic to $X$ by Theorem~\ref{conjecture}~(iii), and $P$ is a  Gleason part for $B$ by Theorem~\ref{conjecture}~(vii).  
Theorem~\ref{conjecture}~(v) gives that $B|P$ is isometrically isomorphic to $C(X)$.  Consequently, there are no analytic discs in $P$.  Furthermore, Theorem~\ref{conjecture} (vii) gives that each point of $\mb\bs P$ is a one-point Gleason part for $B$.  It follows that there are no analytic discs in $\mb$.  
Note that $P$ is contained in $\mb\bs K$ since  $S$ is disjoint from $X \times \partial D$. 
This completes the proof of the theorem in the case when $X$ is compact.

%%%%%%%%%%%%%%**************************

We now treat the more complicated case when $X$ is noncompact.  
Let $\tX$ be a compactification of $X$, and let $\Lambda=\tX \setminus X$. 
Let $Z=\oD^\Lambda$ be the product of $\Lambda$ copies of the closed unit disc, and let $A_\Lambda$ denote the tensor product of $\Lambda$ copies of the disc algebra $A(D)$.  
Then $\mathfrak{M}_{A_\Lambda}=Z$. Denote by $\theta$ the point in $Z$ all of whose coordinates are zero, and let $\Pi_{\theta}$ denote the Gleason part containing 
$\theta$.  Then as shown in \cite{Garnett} (or see \cite[p.~193]{Stout})

\begin{equation*}
\begin{split}
\Pi_{\theta}=\{z=(z_x)_{x\in \Lambda} \in Z: {\rm there\ exists\ } a<1\ {\rm such\ that}\ 
|z_x|<a\phantom{and and a}\\ {\rm for\ all\ } x\in \Lambda\}.
\end{split}
\end{equation*}
Let 
\begin{equation*}
\begin{split}
A=\{f\in C(\tX \times Z): f|({\{x\}\times Z}) \in A_\Lambda\ {\rm for\ all\ } x\in \tX\ 
{\rm and\ }\phantom{and}\\ f|({\tX \times \{ \theta \}})\ {\rm is\ constant} \}.
\end{split}
\end{equation*}
Then $\ma$ is the quotient space obtained from $\tX \times Z$ by collapsing the set $\tX \times \{ \theta \}$ to a point.  We will denote points of $\ma$ by their representatives in $\tX \times Z$.  Then the set  
$$
\Pi=\{(x,z)\in \ma:z \in \Pi_{\theta} \}
$$
is a Gleason part for $A$. 

Write $X=\bigcup_{n=1}^{\infty}X_n$, where $X_1\subset X_2\subset \dots $ and each $X_n$ is compact. Given $y\in\Lambda=\tX \setminus X$, there is a function $h_y \in C(\tX)$  with range contained in $[\frac{1}{2},1]$ such that $h_y(y)=1$ and $\|h_y\|_{X_n}\leq 1-2^{-n}$ for each $n\in \N$. Define a map $\rho : \tX \rightarrow \ma$ by $\rho (x)=(x,H(x))$, where $H(x)=(h_y(x))_{y\in \Lambda}$.  The map $\rho$ is an embedding of $\tX$ into $\ma$. Set $S=\rho(\tX)$, and note that $S \cap \Pi=\rho(X)$.

The set $S$ is a hull in $\ma$; if for each $y\in \Lambda$ we define $g_y$ on $\ma$ by
$$
g_y(x,z)=\bigl(h_y(x)-z _y\bigr)\big/\bigl(3h_y(x)-z_y\bigr),
$$
then each $g_y$ is in $A$ and\vadjust{\kern 4pt} 
$S=\bigcap\limits_{y\in \Lambda} g_y^{-1}(0)$. 

The algebras $A|(S\cap \Pi)$ and $C(\tX)$ are isometrically isometric; indeed, given 
$f\in C(\tX)$, and $y\in \Lambda$, the function $l$ on $\ma$ given by
$$
l(x, z)=z_y f(x) / h_y(x)
$$
is in $A$, and $f=l\circ \rho$.  Consequently, the map $T:A \rightarrow C(\tX)$ given by $T(l)=l\circ\rho$
induces an isometric isomorphism of $A|(S\cap \Pi)$ onto $C(\tX)$.

We claim that distinct points of $S\setminus \Pi$ lie in distinct Gleason parts.  To see this, let $a$ and $b$ be distinct points of $\tX \setminus X$, and choose a function
$f\in C(\tX)$ such that $f(a)=1, f(b)=0$, and $\|f\|=1$. Then the function on $\ma$ given by
$h(y,z)=z_a f(y)$ belongs to $A$ and satisfies $h(\rho(a))=1, h(\rho(b))=0$, and 
$\|h\|=1$.  Thus $\rho(a)$ and $\rho(b)$ lie in different Gleason parts by Lemma~\ref{samepart}.

Note that every function in $A_\Lambda$ takes its maximum modulus on $(\partial D)^{\Lambda}$, and consequently, every function in $A$ takes its maximum modulus on $\tX \times (\partial D)^{\Lambda}$.  Therefore, $A|(\tX \times (\partial D)^{\Lambda})$ is a uniform algebra isometrically isomorphic to $A$.

Now we apply Theorem~\ref{conjecture} to the algebra 
$A|(\tX \times (\partial D)^{\Lambda})$.  Letting $A^C$, $Y^C$, and $\pi$ denote the objects given by Theorem~\ref{conjecture}, set $B=A^C$, $K=Y^C$, and $P=\pi^{-1}(S\cap \Pi)$.  
Then by Theorem~\ref{conjecture}, $P$ is a  Gleason part for $B$ homeomorphic to $X$, and 
$B|P$ is isometrically isomorphic to $C(\tX)$.  Consequently, there are no analytic discs in $P$.  Furthermore, Theorem~\ref{conjecture}~(vii), together with the fact that distinct points of $S\bs \Pi$ lie in distinct Gleason parts for $A$, gives that each point of $\mb\bs P$ is a one-point Gleason part for $B$.  It follows that there are no analytic discs in $\mb$.  Note that $P$ is contained in $\mb\bs K$, since $S\cap\Pi$ is disjoint from $\tX \times (\partial D)^{\Lambda}$.
\end{proof}

%%%%%%%%%%%%%%%%%%%%%%%%%%%%%%%%%%%%%
%%%%%%%%%%%%%%%%%%%%%%%%%%%%%%%%%%%%%

%The metrizable case

%%%%%%%%%%%%%%%%%%%%%%%%%%%%%%%%%%%%%
%%%%%%%%%%%%%%%%%%%%%%%%%%%%%%%%%%%%%

\begin{proof}[Proof of Theorem~\ref{T:2}]
The case when $X$ is compact has really already been proven in the proof of Theorem~\ref{T:1} above, for in that case the uniform algebra $B$ constructed there is defined on a metrizable space in view of the last sentence of 
Theorem~\ref{conjecture}.
In the case when $X$ is noncompact, first note that since $X$ is separable and metrizable,  $X$ can be embeded in $[0,1]^\omega$ (by the standard proof of the Urysohn metrization theorem) and thus does have a metrizable compactification. Let $\tX$ be any such compactification of $X$. 
Now the only change needed in the construction given in the proof of the noncompact case of Theorem~\ref{T:1} is that we must show that we can replace \vadjust{\kern 1pt}the possibly \emph{uncountable} index set 
$\Lambda=\tX \bs X$ by a suitable \emph{countable} subset of $\tX\bs X$.

Write $X=\bigcup_{n=1}^{\infty}X_n$, where $X_1\subset X_2\subset \dots $ and each $X_n$ is compact. Given $y\in \tX \setminus X$, there is a function $h_y \in C(\tX)$  with range contained in $[\frac{1}{2},1]$ such that $h_y(y)=1$ and $\|h_y\|_{X_n}\leq 1-2^{-n}$ for each $n\in \N$. For each $m\in \N$, the collection of sets
$$
\Bigl\{ \{h_y>1-\tfrac{1}{m}\}: y\in \tX \setminus X \Bigr\}
$$
is an open cover of $\tX \setminus X$.
Choose a countable subcover, and let $\{y_k^m: k\in \N\}$ be the set of corresponding $y$'s. 
The proof can now be carried out by repeating the argument given in the noncompact case of Theorem~\ref{T:1} setting $\Lambda=\{y_k^m: k,m\in \N\}$ instead of setting $\Lambda=\tX\bs X$.  Metrizability of $\ma$ is assured by the following lemma, and then metrizability of the space $K=Y^C$ obtained in the final paragraph of the proof is given by the last sentence of Theorem~\ref{conjecture}.
\end{proof}

\begin{lemma}\label{metrizable}
Let $X$
be a compact metrizable space and $Y$ be a quotient space of $X$.  If $Y$ is Hausdorff, then $Y$ is metrizable.
\end{lemma}

This lemma is essentially~\cite[Corollary~26.16]{Jameson}, and a short proof can be found in \cite[Lemma~2.5]{FeinsteinI}.

%%%%%%%%%%%%%%%%%%%%%%%%%%%%%%%%%%%%%
%%%%%%%%%%%%%%%%%%%%%%%%%%%%%%%%%%%%%

% Section

%%%%%%%%%%%%%%%%%%%%%%%%%%%%%%%%%%%%%
%%%%%%%%%%%%%%%%%%%%%%%%%%%%%%%%%%%%%

\section{Gleason parts in polynomial hulls}\label{poly-hulls}

In this section we prove Theorem~\ref{T:3} which shows that every locally compact subspace of a Euclidean space  occurs as a Gleason part in a polynomial hull without analytic discs.  In the proofs of Theorem~\ref{T:1} 
and~\ref{T:2}, we initially constructed, as in Garnett's proof, a uniform algebra $A$ whose maximal ideal space contains analytic discs, and we then eliminated the analytic discs by passing to another uniform algebra $B$ using the Cole construction.  As mentioned in the introduction, we do not see a way to carry out this passage from $A$ to $B$ using the Cole construction so as to obtain a finitely generated uniform algebra.  Therefore, in the proof of 
Theorem~\ref{T:3}, we change the construction of the initial uniform algebra 
$A$; in place of the tensor product of disc algebras used in Garnett's construction, we use a certain special \lq\lq Swiss chesse algebra\rq\rq.  In this way we obtain at the outset a uniform algebra $A$ whose maximal ideal space contains no analytic discs.  We then carry out the passage from $A$ to the final desired uniform algebra $B$ via a modification of the method used by Garnett.  Both the construction of the initial uniform algebra $A$ and the passage from $A$ to $B$ rely heavily on the first author's papers~\cite{Izzo} and~\cite{Izzo3}.

Following Garth Dales and Feinstein~\cite{DalesF}, we will say that a uniform algebra 
$A$ has dense invertibles if the invertible elements of $A$ are dense in $A$.  Dales and Feinstein~\cite{DalesF} constructed a compact set $X$ in $\C^2$ such that 
$\h X$ is strictly larger than $X$ and yet the uniform algebra $P(X)$ has dense invertibles.  As noted in \cite{DalesF} the condition that $P(X)$ has dense invertibles is strictly stronger than the condition that $\h X$ contains no analytic discs.  We will also use the observation in \cite{DalesF} that when $P(X)$
has dense invertibles, the polynomial and rational hulls of $X$ coincide. 
For our construction we will need a special $P(X)$ with dense invertibles provided by the following result of the first author.

\bthm \label{Cor8.3} \cite[Corollary 8.3]{Izzo3}
Let $\Omega \subset \C^N$ {\rm(}$N\geq2${\rm)} be any bounded open set, and let $x_0\in \Omega$. Then there exists a compact set $X\subset \partial \Omega$ such that $x_0$ is in $\hh X$ and is a one-point Gleason part for $P(X)$, and $P(X)$ has dense invertibles.
\ethm

In Garnett's construction, the passage from the initial uniform algebra $A$ to the desired uniform algebra $B$ uses the big disc algebra.  Applying the above theorem, we obtain in the following result a suitable replacement for the big disc algebra in our context.

\bthm \label{U_a}
There exists a compact set $X\subset \partial D^2$ such that 
\item{\rm(i)} $0$ is in $\h X$ and is a one-point Gleason part for $P(X)$,
\item{\rm(ii)} $P(X)$ has dense invertibles, and
\item{\rm(iii)} there is a subset $J$ of $X$ such that $0\notin\h J$
and for every representing measure $\mu$ for $0$ as a functional on $P(X)$, we have
$\mu (J)>0$.
\ethm

\begin{proof}
Applying Theorem~\ref{Cor8.3} with $N=2, \Omega=D^2$, and $x_0=0$, we get a compact set $X\subset \partial D^2$ such that (i) and (ii) hold.  
Now let $\mu$ be a representing measure for $0$.
Set $J=X \cap\bigl((\pD)\times \oD\bigr)$ and $K=X\cap \bigl(\oD\times (\pD)\bigr)$.  Because $P(X)$ has dense invertibles, the same is true of $P(J)$ and $P(K)$.  Thus $\h J=\hr J$ and $\h K=\hr K$.  Furthermore the sets $\pD \times \oD$ and $\oD \times \pD$ are each rationally convex, so we conclude that $\h J \subset \pD \times \oD$ and $\h K\subset \oD \times \pD$.  In particular neither $\h J$ nor $\h K$ contains $0$.  Consequently, the representing measure $\mu$ is supported on neither $J$ nor $K$.  Since $X=J \cup K$, it follows that each of $J$ and $K$ has positive $\mu$-measure.
\end{proof}

The next theorem is a variation on \cite[Theorem~1]{Garnett} of Garnett. 

\bthm\label{basic}
Let $A$ be a uniform algebra generated by $n$ elements
such that $\ma$ contains no
analytic discs.  Let $S\subset \ma$ be a hull, and let $\Pi$ be a Gleason part for $A$ with $S\cap \Pi \neq \emptyset$. Then there exists a uniform algebra $B$ on a compact Hausdorff space $Y$ such that 
\item{\rm(i)} $B$ is generated by $n+2$ elements,
\item{\rm(ii)} $\mb$ contains no analytic discs,
\item{\rm(iii)} there is a Gleason part $Q$ for $B$ homeomorphic to $S\cap \Pi$ with $Q\subset \mb\bs Y$, and
\item{\rm(iv)} $B|Q$ is isometrically isometric to an algebra lying between $A|(S\cap \Pi)$ and its uniform closure in the space of bounded continuous functions on $S\cap \Pi$.
\ethm

\begin{proof}
Let $X$ and $J$ be as in Theorem~\ref{U_a}.  Let $g_1,\dots, g_n$ be generators for $A$.  The algebra $P(X)\otimes A$, regarded as a uniform algebra on its maximal ideal space ${\mathfrak{M}_{P(X)\otimes A}}=\h X\times \ma$, is generated by the 
$n+2$ functions $z_1\circ \pi_1, z_2\circ \pi_1, g_1\circ \pi_2, \dots , g_n\circ \pi_2$, where $\pi_1$ and $\pi_2$ are the projections of $\h X \times \ma$ onto the first and second factors, respectively.  Set 
$$
Y=(J\times \ma)\cup (X \times S),
$$ 
and let $B$ be the closure in $C(Y)$ of the restriction  algebra 
$(P(X)\otimes A)|Y$. 

The maximal ideal space $\mb$ of $B$ is the $(P(X)\otimes A)$-convex hull $\h Y$ of $Y$ defined by 
\[
\h Y = \{ x\in \h X \times \ma: |f(x)| \leq \|f\|_Y  \ {\rm for\ all\ } f\in P(X)\otimes A\}.
\]
We now show that 
\begin{equation}\label{hatY}
\h Y = (\h J\times \ma)\cup (\h X \times S).
\end{equation}
First note that for each function $f\in P(X)\otimes A$ and $s\in \ma$, the function on $\h X$ given by $x\mapsto f(x,s)$ is in $P(X)$.  It follows that $\h Y$ contains the set on the right hand side of (\ref{hatY}).  
Now suppose $x=(x_1,x_2)$ is in $(\h X \times \ma)$ but is not in the set on the right hand side of (\ref{hatY}).  Then $x_1\notin \h J$ and $x_2\notin S$.  Since $S$ is a hull, there exists $g\in A$
such that $g|S=0$ and $g(x_2)=1$.  Since $x_1\notin \h J$, there exists 
$f\in P(X)$ such that $f(x_1)=1$ and $\|f\|_J<1$.  Replacing $f$ by a sufficiently high power of $f$, we may assume that $\|f\|_J<\|g\|^{-1}$. 
Then the function $h$ on $\h X\times \ma$ defined by 
$h(u_1,u_2)=f(u_1)g(u_2)$ is in $P(X)\otimes A$ and satisfies that $h(x)=1$ and $\|h\|_Y<1$.  Thus $x$ does not belong to $\h Y$. 

Note that $B$ is generated by the $n+2$ functions $(f_1\circ \pi_1)|Y, (f_2\circ \pi_1)|Y, (g_1\circ \pi_2)|Y, \dots , (g_n\circ \pi_2)|Y$. Also $\mb$ 
contains \vadjust{\kern 1pt}no analytic discs, since neither 
$\h X$ nor $\ma$ contains analytic discs.

Note that the set 
$$
Q=\{0\}\times (S \cap \Pi),
$$
which is clearly homeomorphic to $S\cap \Pi$, is contained in $\h Y \bs Y = \mb\bs Y$.  We show that it is a Gleason part for $B$.  For $s\in S\cap \Pi$, let $p_s=(0,s)\in Q$.  Suppose
$x=(x_1,x_2)\in \h Y\setminus Q$.  If $x_1\neq 0$, then since 0 is a one-point Gleason part for $P(X)$, using functions from $P(X)$ we get that $x$ and $p_s$ lie in different Gleason parts.  Similarly, if $x_2\notin \Pi$, then again $x$ and $p_s$ lie in different Gleason parts.  If finally $x_2\notin S$, then we see from (\ref{hatY}) that $x_1\in \h J$, so $x_1\neq 0$, which we already noted implies that $x$ and $p_s$ lie in different Gleason parts. Hence $Q$ is a union of Gleason parts. To show that $Q$ \emph{is} a Gleason part, let $\mu$ be a representing measure on $X$ for the functional evaluation at 0 on $P(X)$.  For $s\in S\cap \Pi$ and $g\in B$, we have
$$
g(p_s)=g(0,s)=\int_X g(x,s)d\mu(x).
$$
For $s,t\in S\cap \Pi$ and $g\in B$ such that $\|g\| \leq 1$, we have 
\begin{equation}\label{inequality}
\begin{split}
\bigl |g(p_s)-g(p_t)\bigr |&\leq \int_X \bigl |g(x,s)-g(x,t)\bigr |d\mu (x)\\ 
&\leq 2\mu (X\setminus J)+ \int_J\bigl |g(x,s)-g(x,t)\bigr |d\mu(x).
\end{split}
\end{equation}
For $x\in J$ we have $\{x\}\times \ma \subset Y$, and thus the function on $\ma$ given by $y\mapsto g(x,y)$ is in $A$ and of norm at most $1$. Since $s$ and $t$ are in the common Gleason part $\Pi$, there is a constant $c<2$, independent of $x$, such that $|g(x,s)-g(x,t) |<c$. Since $\mu (X)=1$ and $\mu (J)>0$, we get from (\ref{inequality}) that 
$|g(p_s)-g(p_t)|< 2\mu (X\setminus J)+ c\mu(J) <2$, 
and hence $p_s$ and $p_t$ lie in a common Gleason part.  Thus $Q$ is a Gleason part.

Finally, we note that $(P(X)\otimes A)|Q$ is isometrically isometric to $A|(S\cap \Pi)$ since for each function $f \in P(X)\otimes A$, the function on $\ma$ given by $y\mapsto f(0,y)$ is in $A$.  Assertion (iv) of the theorem follows.
\end{proof} 

Give $a\in \C$ and $r>0$, we will denote 
the open disc of radius $r$ with center $a$ by $D(a,r)$ and the corresponding closed disc by $\od(a,r)$. 
For $X\subset \C^N$ a compact set, and for $\Omega \subset \C^N$ an open set each of which contains the origin, we will denote by  $\mathscr{B}_0(X)$ the set of rational functions $f$ holomorphic on a neighborhood of $X$ such that $f(0)=0$ and $\|f\|_X\leq 1$, and by $\mathscr{B}_0(\Omega)$ the set of functions $f$ holomorphic on $\Omega$ such that $f(0)=0$ and $\|f\|_{\Omega} \leq 1$. 
The next two lemmas are taken from the first author's papers \cite{Izzo} and \cite{Izzo4}. 

\begin{lemma}\cite[Lemma~3.4]{Izzo4} \label{R}
Suppose $K$ is a compact set containing the origin and contained in an open set $\Omega\subset \CN$.  Then there exists an $R$ with $0<R<1$ such that $\|f\|_K\leq R$ for all $f\in \Bzero \Omega$ if and only if $K$ is contained in the component of $\Omega$ that contains the origin.
\end{lemma}
\begin{lemma}\cite[Lemma~5.3]{Izzo}\label{R+epsilon}
Let $K$ be a compact set containing the origin and contained in an open set $\Omega\subset \C$, let $a\in \Omega\setminus K$, and let $\varepsilon>0$ be given.  Let $0<R<1$, and suppose that $\|f\|_K\leq R$ for all $f\in \Bzero \Omega$.
Then there exists an $r>0$ such that $\odisc ar{} \subset \Omega\setminus K$ and $\|f\|_K\leq R+\varepsilon$ for all $f\in \Bzero {\Omega\setminus \odisc ar{}}$.
\end{lemma}

By a \textit{Swiss cheese} we shall mean a 
compact set $L$ obtained from the closed unit disc 
$\od\subset \C$ by deleting a sequence of open discs $\{D_j\}_{j=1}^\infty$ with radii 
$\{r_j\}_{j=1}^\infty$ and contained in $D$ such that the closures of the $D_j$ are disjoint, 
$\sum_{j=1}^\infty r_j<\infty$, and the resulting set $L=\oD\setminus (\bigcup_{j=1}^{\infty}D_j)$ has no interior.
It is well known that every Swiss cheese $L$ satisfies $R(L)\neq C(L)$.

The proof of the next result is reminiscent of, though much simpler than, the proof of \cite[Theorem~1.5]{Izzo}.

\bthm\label{swiss}
There exists a Swiss cheese $L$ such that $[0,1)$ is contained in a single Gleason part for $R(L)$.
\ethm

\begin{proof}
Set $D'=D\bs [0,1)$ and $I_n=[0, 1-2^{-n}]$ for each $n\in \N$.  

Choose a sequence $\{ \alpha _j\}$ that is dense in $D'$.  Set $R_1=1/2$.  By the Schwarz Lemma $\|f\|_{L_1}\leq R_1$, for every $f\in \mathscr{B}_0(D)$. 
Set $a_1=\alpha_1$.
By Lemma~\ref{R+epsilon}, there exists $0<r_1<1/2$ such that $\odisc ar1\subset D'$, the boundary of $\disc ar1$ is disjoint form the set $\{\alpha_j\}$, and
$$
\|f\|_{I_1}\leq R_1 + \tfrac 14 (1-R_1) \quad \forallBzero {D\setminus \odisc ar1}.
$$
By Lemma~\ref{R}, there exists $0<R_2<1$ such that
$$
\|f\|_{I_2}\leq R_2 \quad \forallBzero {D\setminus \odisc ar1}.
$$
Now let $a_2$ be the first term of the sequence $\{\alpha_j\}$ not belonging to $\disc ar1$.  
By Lemma~\ref{R+epsilon}, there exists $0<r_2<1/4$ such that $\odisc ar2\subset D'\bs \odisc ar1$, the boundary of $\disc ar2$ is disjoint from the set $\{\alpha_j\}$, and 
$$
\|f\|_{I_1}\leq R_1 + (\textfrac 14 + \textfrac 18) (1-R_1) \quad \forallBzero {D\setminus [ \odisc ar1 \cup \odisc ar2]}
$$
and
$$
\|f\|_{I_2}\leq R_2 + \textfrac 14 (1-R_2) \quad \forallBzero {D\setminus [ \odisc ar1 \cup \odisc ar2]}
$$

We continue by induction.
Suppose that for some $k\geq 2$, we have chosen $a_1,\ldots, a_k$,  $r_1,\ldots, r_k$, and $R_1,\ldots, R_k$ such that for all $1\leq l\leq k$ and all $1\leq n\leq m\leq k$
the following conditions hold:
\begin{enumerate}
\item[(i)]  $a_l$ is the first term of $\{\alpha_j\}$ not in $\disc ar1\cup\cdots\cup \disc ar{l-1}$ 
\item[(ii)] $0<r_l<1/2^l$
\item[(iii)] $\odisc aul\subset D'\setminus [\odisc ar1 \cup\cdots\cup \odisc ar{l-1}]$
\item[(iv)] $\partial \disc arl$ is disjoint from the set $\{\alpha_j\}$
\item[(v)] we have the inequality
$$\|f\|_{I_n}\leq R_n +(\textfrac 14 + \textfrac 18 + \cdots + \textfrac 1{2^{m-n+2}})(1-R_n) \quad \forallBzero {D\setminus \textstyle\bigcup\limits_{j=1}^m \odisc arj}.$$
\end{enumerate}
Let $a_{k+1}$ be the first term of $\{\alpha_j\}$ not in $\disc ar1 \cup\cdots\cup \disc ark$.  By Lemma~\ref{R}, there exists $0<R_{k+1}<1$ such that
$$
\|f\|_{I_k}\leq R_k  \quad \forallBzero {D\setminus \textstyle\bigcup\limits_{j=1}^k \odisc arj}.
$$
By Lemma~\ref{R+epsilon}, there exists $0<r_{k+1}<1/2^{k+1}$ such that 
conditions (i)--(v) continue to hold when, in the restriction $1\leq n\leq m\leq k$,  we replace $k$ by $k+1$. 
Thus the induction can proceed.

We conclude that there are sequences $\{a_j\}$, $\{r_j\}$, and $\{R_j\}$ such that for all $1\leq l$ and all $1\leq n\leq m$ conditions (i)--(v) hold.  

Set $L=\od\bs \bigcup_{j=1}^\infty \disc arj$.

Now let $z\in [0,1)$ be arbitrary.  Then $z\in I_n$ for some $n\in \N$.  
Given $g\in \Bzero L$, there is some $m\geq n$ such that $g\in \Bzero {D\setminus \bigcup_{j=1}^m \odisc arj}$.  By condition (v),
$$|g(z)|\leq R_n +(\textfrac 14 + \textfrac 18 + \cdots + \textfrac 1{2^{m-n+2}})(1-R_n).$$  Therefore,
$$\sup_{f\in \Bzero L} |f(z)| \leq R_n + \textfrac 12 (1-R_n) <1.$$
Thus $z$ is in the same Gleason part for $R(L)$ as $0$ by Lemma~\ref{samepart}.
\end{proof}

\begin{lemma}\label{non-unital}
Suppose that $A$ is a uniform algebra generated by $m$ functions on a compact Hausdorff space $X$, that $C$ is a uniform algebra generated by $n$ functions on a compact Hausdorff space $Y$, and that $p\in Y$.  Then
$$
\{f\in A\otimes C: f|({X\times\{p\}})\ {\rm is\ constant} \}
$$
is a unital Banach algebra generated by $mn$ functions.
\end{lemma}

\begin{proof}
Let $f_1,\dots ,f_m$ be generators for $A$ and $g_1,\dots, g_n$ be generators for $C$.  Replacing each $g_j$ by $g_j-g_j(p)$, we may assume that each $g_j$ vanishes at $p$.  
Let
$$
C_0=\{g\in C: g(p)=0\}.
$$
Then $C_0$ is the smallest non-unital Banach algebra containing the functions $g_1,\dots, g_n$. Denote by $(A\otimes C_0) \oplus \C$ the smallest Banach algebra of functions on $X\times Y$ that contains the constants and the functions $f\otimes g$ defined by $(f\otimes g)(x,y) = f(x)g(y)$ for $f\in A$ and $g\in C_0$. Then
$$
(A\otimes C_0)\oplus \C=\{f\in A\otimes C: f|({X\times\{p\}})\ {\rm is\ a\ constant} \}.
$$
The proof is thus concluded by noting
that $(A\otimes C_0)\oplus \C$ is generated by the set $\{ f_k \otimes g_l:  k=1,\dots ,m {\rm\ and\ } l=1,\dots ,n\}$.
\end{proof}

The proof of Theorem~\ref{T:3} will use the fact that a locally compact Hausdorff space is open in each of its compactifications.  This fact is surely known, but we have not seen it in the standard topology texts, so we include a proof.

\bthm \label{open}
For a completely regular space, the following are equivalent:
\item{\rm (a)} $X$ is locally compact
\item{\rm (b)} $X$ is open in some compactification of $X$
\item{\rm (c)} $X$ is open in every compactification of $X$.
\ethm

\begin{proof}
The implication (c) $\implies$ (b) is immediate from the existence of a compactification of a completely regular space, and the implication 
(b) $\implies$ (c) is well known (and easily proven).  That (a) $\implies (c)$ is immediate from the following lemma.
\end{proof}

\blem \label{l.c.}
Let $X$ be a dense subspace of a Hausdorff space $Y$.  If $X$ is locally compact at a point $x\in X$, then $x$ is an interior point of $X$ relative to $Y$.
\elem

\bpf
Suppose that $X$ is locally compact at $x$.  Choose an open neighborhood $U$ of $x$ in $X$ with the closure $\overline U^X$ of $U$ in $X$ compact.  Then $\overline U^X$ must be closed in $Y$.  Consequently, $\overline U^X$ equals the closure $\overline U$ of $U$ in $Y$.  Since $U$ is open in $X$, we have $U=X\cap W$ for some open set $W$ of $Y$.  Then $W\setminus {\overline U}$ is an open set of $Y$ that is disjoint from $X$.  Thus $W\setminus {\overline U}$ must be empty, since $X$ is dense in $Y$.  Thus we have $W\subset {\overline U} = \overline U^X \subset X$.  Thus $W=U$, so $W$ is an open set of $Y$ such that $x\in W\subset X$.  Thus $x$ is an interior point of $X$ relative to $Y$.
\epf

We are now ready to prove Theorem~\ref{T:3}

\begin{proof}[Proof of Theorem~\ref{T:3}]
Let $\tX$ be a compactification of $X$ contained in $\R^m$.  (Note that because $X$ is homeomorphic to a \emph{bounded} subspace of $\R^m$, such a compactification necessarily exists.  If $X$ is compact, then $\tX=X$.)
Let $L$ be the Swiss cheese of Theorem~\ref{swiss} and consider the algebra
$$
A=\{f\in C(\tX)\otimes R(L): f|({\tX \times \{ 0 \}})\ {\rm is\ constant} \}.
$$
Then $\ma$ is the quotient space obtained from $\tX \times L$ by collapsing the set $\tX \times \{0\}$ to a point. 
We will denote points of $\ma$ by their representatives in $\tX \times L$. 
Then the set $\{(x,y)\in \ma: y\in [0,1)\}$ is contained in a single Gleason part $\Pi$ of $A$.  Lemma~\ref{peak} shows that no point of the form $(x,1)\in \ma$ lies in $\Pi$.

Since $X$ is locally compact, $\tX \setminus X$ is closed in $\tX$ by 
Theorem~\ref{open}, and there exists a function $h \in C(\tX)$  with range contained in $[\frac{1}{2},1]$ such that $h(\tX \setminus X)=1$ and $h(x)<1$ for all $x\in X$. 
(If $X$ is compact, simply take $h(x)=1/2$ for all $x\in X$.) 
Define $\rho :\tX \rightarrow \ma$ by 
$\rho(x)=(x,h(x))$,
and observe that $\rho$ is an embedding. The set $S=\rho(\tX)$ is a hull since the function $g$ defined on $\ma$ by $g(x,y)=1-\bigl(y/h(x)\bigr)$ is in $A$ and $S=g^{-1}(0)$. Furthermore, $\rho(X)=S\cap \Pi$. 

Note that $\ma$ contains no analytic discs, since neither 
${\mathfrak{M}}_{C(\tX)}$ nor ${\mathfrak{M}}_{R(L)}$ contains analytic discs.

It is well known that $R(L)$ is generated by 2 functions (see the last paragraph of Section~\ref{prelim}) and $C(\tX)$ is generated by $m$ functions.  Thus Lemma~\ref{non-unital} gives that $A$ is generated by $2m$ functions.

The algebras $A|(S\cap \Pi)$ and $C(\tX)$ are isometrically isometric; indeed, given 
$f\in C(\tX)$, the function $l$ on $\ma$ given by
$$
l(x, z)=z f(x) / h(x)
$$
is in $A$, and $f=l\circ \rho$.  Consequently, the map $T:A \rightarrow C(\tX)$ given by $T(l)=l\circ\rho$
induces an isometric isomorphism of $A|(S\cap \Pi)$ onto $C(\tX)$.

Now let the uniform algebra $B$ and the compact Hausdorff space $Y$ be obtained from $A$ as in Theorem~\ref{basic}.  Then $B$ is generated by $2m+2$ functions, say $f_1,\ldots, f_{2m+2}$.  Let $\tau:Y\rightarrow \C^{2m+2}$ be defined by $\tau(y)=\bigl(f_1(y),\ldots, f_{2m+2}(y)\bigr)$. Then the set $K=\tau(Y)\subset\C^{2m+2}$ has all the required properties.
\end{proof}

We conclude with examples showing that
in contrast to the situation in Theorems~\ref{T:1} and~\ref{T:2}  which provided complete characterizations of the spaces that occurred as Gleason parts in the given contexts, local compactness does not characterize the subspaces of a Euclidean space that occur as Gleason parts in polynomial hulls.  We begin with two examples of compact planar sets $L$ such that $R(L)$ has a Gleason part that is not locally compact.  These can be translated into examples involving algebras of the form $P(K)$, with $K$ in $\C^2$, by the standard device recalled at the end of Section~\ref{prelim}, but in both cases, the relevant Gleason part is not contained in the complementary set $\h K \bs K$.  For our final example, we in essence take the tensor product of such a $P(K)$ with $P(X)$ for $X$ as in Theorem~\ref{U_a} to get an example with the Gleason part in the complementary set, thus showing that local compactness of the space $X$ in Theorem~\ref{T:3} is not a necessary condition for the conclusion of (at least the first part of) the theorem to hold.

\begin{example}
Let $L\subset \C$ be the ``road runner set" obtained by deleting from the closed unit disc $\oD$ a sequence of open discs $\{ D_j \}_{j=1}^{\infty}$ with radii $r_j=9^{-j}$ and centers $c_j=3^{-j}+9^{-j}$. It is shown in \cite[Lemma 11.1]{GamelinR} that $R(L)$ has a single nontrivial  Gleason part $P$, namely the interior of $L$ union the origin.  Lemma~\ref{l.c.} shows that $P$ is not locally compact at $0$ since $\overline{P}=L$ but $0$ is not an interior point of $P$ relative to $L$. 
\end{example}

\begin{example}\label{example}
It is proven in \cite[Theorem 1.5]{Izzo} that there exists a Swiss cheese $L=\oD\setminus (\bigcup_{j=1}^{\infty}D_j)$ with a Gleason part $P$ for $R(L)$ of full measure in $L$. We claim that $P$ is locally compact at no point. 
As shown in \cite[Remark~5.6]{Izzo}, the fact that $P$  has full measure in the Swiss cheese $L$ implies that $P$  is dense in $L$.  Applying Lemma~\ref{peak} (for instance) shows that each point of 
$\bigcup_{j=1}^{\infty}\partial D_j$, a dense subset of $L$, is a one-point part, so $P$ has empty interior in $L$.  The claim now follows from Lemma~\ref{l.c.}.
\end{example}

\begin{example}
Let $L$ be a Swiss cheese
such that $R(L)$ has a Gleason part $P$ that is not locally compact (for instance the Swiss cheese of either Example 1 or 2).  Let $Y=\tau(L)$, where $\tau$ is as in the paragraph at the end of Section~\ref{prelim}. 
Let $X$ be as in Theorem~\ref{U_a}.  Then $X\times Y$ is a compact set in $\C^4$ such that $\h {X\times Y}= \h X \times \h Y$ contains no analytic discs, and $\{0\} \times \tau(P)$ is a Gleason part for $P(X\times Y)$ that is not locally compact and is contained in $\h {X\times Y} \bs (X\times Y)$.
\end{example}


\begin{thebibliography}{BC87}

\bibitem{Basener} R. F. Basener, {\it On rationally convex hulls},
Trans.\ Amer.\ Math.\ Soc.\ {\bf 182} (1973), 353--381.

\bibitem{Browder}
A. Browder, 
{\it Introduction to Function Algebras\/},
Benjamin, New York, 1969.

\bibitem{Cole} B. J. Cole, {\it One-point parts and the peak point conjecture\/}, Ph.D. dissertation, Yale University, 1968.

\bibitem{CGI} B. J. Cole, S. N. Ghosh, and A. J. Izzo, {\it A hull with no nontrivial Gleason parts\/}, Indiana Univ.\ Math.\ J.\ {\bf 67} (2018), 739--752.

\bibitem{DalesF}
H. G. Dales and J. F. Feinstein, {\it Banach function algebras with dense invertible group\/}, Proc.\ Amer.\ Math.\ Soc. {\bf 136\/} (2008), 1295--1304.

\bibitem{Feinstein1} J. F. Feinstein, {\it A nontrivial strongly regular uniform algebra\/}, J.\ London Math.\ Soc.\ {\bf 45} (1992), 288--300.

\bibitem{Feinstein2} J. F. Feinstein, {\it Regularity conditions for Banach function algebras\/}, in {\em Function Spaces}, Lecture Notes in Pure and Applied 
Math.\ (ed.\ K. Jarosz) {\bf 172} (1995), 117--122.

\bibitem{Feinstein3} J. F. Feinstein, {\it A counterexample to a conjecture of S. E. Morris\/}, Proc.\ Amer.\ Math.\ Soc. {\bf 132\/} (2004), 2389--2397.

\bibitem{FeinsteinH} J. F. Feinstein and M. J. Heath, {\it Regularity and amenability conditions for uniform algebras\/}, in {\em Function Spaces}, Contemporary Mathematics (ed.\ K. Jarosz) {\bf 435} (2007), 159--169.

\bibitem{FeinsteinI} J. F. Feinstein and A. J. Izzo, {\it A general method for constructing essential uniform algebras\/}, Studia Math.\ {\bf 246} (2019), 47--61.

\bibitem{Gamelin} T. W. Gamelin, 
{\it Uniform Algebras}, 2nd ed., Chelsea Publishing Company, New York, NY, 1984.

\bibitem{GamelinR} T. W. Gamelin and H. Rossi, {\it Jensen measures and algebras of analytic functions},  in
{\em Function Algebras}, 
(ed.~F. Birtel), Scott, Foresman and Co., (1966), pp.15--35.

\bibitem{Garnett} J. Garnett, {\it A topological characterization of Gleason parts}, Pacific J. Math.\ {\bf 20\/} (1967), 59--64.


\bibitem{Gleason} A. Gleason,
{\it Function algebras\/}, Seminar on Analytic Functions, vol.~II, Institute for Advanced Study, Princeton (1957), 213--226.

\bibitem{HW} W. Hurewicz and H. Wallman, {\it
Dimension Theory}, revised edition,
Princeton University Press, 1948.

\bibitem{Izzo} A. J. Izzo, {\it Gleason parts and point derivations for uniform algebras with dense invertible group\/}, Trans.\ Amer.\ Math.\ Soc.\ {\bf 370} (2018), 4299-4321.

\bibitem{Izzo2} A. J. Izzo, {\it Spaces with polynomial hulls that contain no analytic discs\/}, Math.\ Ann.\ (accepted).

\bibitem{Izzo3} A. J. Izzo, {\it A doubly generated uniform algebra with a 
one-point Gleason part off its Shilov boundary\/}, Studia Math. (accepted).

\bibitem{Izzo4} A. J. Izzo, {\it Gleason parts and point derivations for uniform algebras with dense invertible group II\/} (submitted).

\bibitem{Jameson} G. J. O. Jameson,
{\it Topology and Normed Spaces}, Chapman and Hall, London, 1974.

\bibitem{Rossi} H. Rossi,
{\it Holomorphically convex sets in several complex variables\/}, Ann.\ of 
Math.\ {\bf 74} (1961), 740--793.

\bibitem{Stol1}  G. Stolzenberg, {\it A hull with no analytic structure\/},
J.\ Math.\ Mech.\  {\bf 12} (1963), 103--111.

\bibitem{Stout} E. L. Stout,
{\it The Theory of Uniform Algebras\/},
Bogden \& Quigley, New York, 1971.

\end{thebibliography}
\end{document}